\documentclass[11pt]{article}
\usepackage{a4wide}
\usepackage[dvipdfmx]{graphicx}
\usepackage{color}
\usepackage{amsmath,amssymb,amsthm}
\usepackage{authblk}
\usepackage{here}
\newtheorem{theorem}{Theorem}
\newtheorem{corollary}[theorem]{Corollary}
\newtheorem{remark}[theorem]{Remark}
\newtheorem{example}[theorem]{Example}

\newcommand{\zb}{\overline{z}}
\newcommand{\aone}{\overline{a_1}}
\newcommand{\atwo}{\overline{a_2}}
\newcommand{\athree}{\overline{a_3}}

\newcommand\s[1]{\sigma_{#1}}
\newcommand\sbar[1]{\overline{\sigma_{#1}}}

\title{Interior and exterior curves of finite Blaschke products}
\author{Masayo Fujimura
        \thanks{This work was partially supported by JSPS KAKENHI 
                Grant Number JP15K04943.}}
\affil{Department of Mathematics, National Defense Academy\\ 
       Yokosuka 239-8686, Japan}

\begin{document}
\maketitle

\begin{abstract}
For a Blaschke product $ B $ of degree $ d $ and $ \lambda $ on
$ \partial\mathbb{D} $, let $ \ell_{\lambda} $ be the set
of lines joining each distinct two preimages in $ B^{-1}(\lambda) $.
The envelope of the family of lines 
$ \{\ell_{\lambda}\}_{\lambda\in\partial\mathbb{D}} $
is called the interior curve associated with $ B $. 
In 2002, Daepp, Gorkin, and Mortini proved the interior curve 
associated with a Blaschke product of degree 3 forms an ellipse.
While let $ L_{\lambda} $ be the set of lines tangent to 
$ \partial{\mathbb{D}} $ at the $ d $ preimages $ B^{-1}(\lambda) $ and 
the trace of the intersection points of each two elements in 
$ L_{\lambda} $ as $ \lambda $ ranges over the unit circle is called the
exterior curve associated with $ B $.
In 2017, the author proved the exterior curve associated with a 
Blaschke product of degree 3 forms a non-degenerate conic.

In this paper, for a Blaschke product of degree $ d $, 
we give some geometrical properties that lie between
the interior curve and the exterior curve.
\end{abstract}


\bigskip


\section{Introduction}
A {\it Blaschke product} of degree $ d $ is a rational function 
defined by
\begin{equation}\label{eq:B}
     B(z)=e^{i\theta}\prod_{k=1}^{d}\frac{z-a_k}{1-\overline{a_k} z} \qquad
             (a_k\in \mathbb{D},\ \theta\in\mathbb{R}). 
\end{equation}
In the case that $ \theta=0 $ and $ B(0)=0 $, 
$ B $ is called {\it canonical}. 

For a Blaschke product \eqref{eq:B} of degree $ d $, set
$$
    f_1(z)=e^{-\frac{\theta}{d}i}z, \quad \mbox{and} \quad
    f_2(z)=\frac{z-(-1)^da_1\cdots a_de^{i\theta}}
                {1-(-1)^d\overline{a_1\cdots a_de^{i\theta}}z}.
$$
Then, the composition $ f_2\circ{B}\circ f_1 $ is a canonical one,
and geometrical properties with respect to preimages 
of these two Blaschke products $ B $ and $ f_2\circ{B}\circ f_1 $
are same.
Hence, we only need to consider a canonical Blaschke product 
for the following discussions.
Moreover, the derivative of a Blaschke product has no zeros on
$ \partial\mathbb{D} $. For instance, see \cite{Inner}.
Hence, there are $ d $ distinct preimages 
of $ \lambda\in\partial\mathbb{D} $ by $ B $.

Let $ z_1,\cdots,z_d $ be the $ d $ distinct preimages of 
$ \lambda\in\partial\mathbb{D} $ by $ B $,
and $ \ell_{\lambda} $ the set of lines joining $ z_j $ and $ z_k $
$ (j\neq k) $.
Here, we consider the family of lines
$$ {\cal L}_B=\{\ell_{\lambda}\}_{\lambda\in\mathbb{D}}, $$
and the envelope $ I_B $ of $ {\cal L}_B $.
We call  the envelope $ I_B $ the {\it interior curve associated with $ B $}.

For a Blaschke product of degree $ 3 $, 
the interior curve forms an ellipse \cite{daepp} 
and corresponds to the inner ellipse of Poncelet's theorem (cf. \cite{flatto}).

\begin{theorem}[U.~Daepp, P.~Gorkin, and R.~Mortini \cite{daepp}]
  \label{thm:DGM} \quad
  Let $ B $ be a canonical Blaschke product of degree $3$ with zeros
  $ 0,\,a_1 $, and $ a_2 $.
  For $ \lambda\in\partial\mathbb{D} $,
  let $ z_1,z_2$, and $ z_3 $ denote the points mapped to 
  $ \lambda $ under $ B $, and write
  \begin{equation}\label{eq:partial}
     F(z)=\frac{B(z)/z}{B(z)-\lambda}
       =\frac{m_1}{z-z_1}+\frac{m_2}{z-z_2}+\frac{m_3}{z-z_3}. 
  \end{equation}
  Then the lines joining $ z_1 $ and $ z_2 $ is
  tangent to the ellipse $ E $ with equation
  \begin{equation}\label{eq:DGM}
      |z-a_1|+|z-a_2|=|1-\aone a_2| 
  \end{equation}
  at the point $ \zeta_3=\dfrac{m_1z_2+m_2z_1}{m_1+m_2} $.
  Conversely, every point of $ E $ is the point of tangency with
  $ E $ of a line that passes through two distinct points $ z_1 $  
  and $ z_2 $ on  the unit circle for which $ B(z_1)=B(z_2) $.
\end{theorem}

This result reminds us of the following classical result in Marden's book
\cite{marden} that was proved first by Siebeck \cite{siebeck}.

\begin{theorem}[Siebeck \cite{siebeck}]\label{thm:marden1}
   The zeros $ z_1' $ and $ z_2' $ of the function 
   $$  F(z)=\frac{m_1}{z-z_1}+\frac{m_2}{z-z_2}+\frac{m_3}{z-z_3} \
       \Big(=\frac{n(z-z_1')(z-z_2')}{(z-z_1)(z-z_2)(z-z_3)}\Big) $$
   are the foci of the conic which touches the line segments 
   $ z_1z_2,\ z_2z_3 $ and $ z_3z_1 $ in the points
   $ \zeta_3,\zeta_1 $, and $ \zeta_2 $ that divide these segments in
   the ratios $ m_1:m_2,\ m_2:m_3 $ and $ m_3:m_1 $, respectively.
   If $ n=m_1+m_2+m_3\neq0 $, the conic is an ellipse or hyperbola 
   according as $ nm_1m_2m_3>0 $ or $ <0  $.
\end{theorem}
 
  For a Blaschke product $ B $ of degree 3, let $ z_1,z_2 $, and $ z_3 $ 
  be the preimages for some  $ \lambda \in\partial\mathbb{D} $ by $ B $
  and $ F $ be defined as \eqref{eq:partial}, then 
  the following folds (\cite[Lemma 4]{daepp}),
  $$ 
      m_1+m_2+m_3=1  \quad \mbox{and} \quad   
      0<m_j<1 \ \mbox{for} \ j=1,2,3.
  $$
  Theorem \ref{thm:DGM} asserts the existence of
  ``the common ellipse'' for a given Blaschke product,
   as long as the given three points $ z_1,z_2 $, and $ z_3 $ are the
   preimage for some $ \lambda $ on the unit circle.

  The ellipse \eqref{eq:DGM} is also related to the 
  numerical range of another specific matrix with eigenvalues 
  $ a_1 $ and $ a_2 $ of the non-zero zero points of $ B $.
  Gorkin and Skubak studied such relations (\cite{gorkin}).
%  It is also known that there are some connection with 
%  Theorem \ref{thm:marden1}
%  and the numerical range of a specific $ 2\times 2 $ matrix
%  (cf. \cite{wu} and \cite{gau}).

Moreover, for a canonical Blaschke product $ B $ of degree $ 4 $ with 
zeros $ 0,a_1,a_2$, and $ a_3 $,
the interior curve associated with $ B $ is defined by the equation of
total degree 6 with respect to $ z $ and $ \zb  $.
The coefficients of $ z^6 $ and $ \zb^6 $ are
$$ 
      (\aone-\atwo )^2(\atwo -\athree )^2
     (\athree -\aone )^2 
     \quad \mbox{and}  \quad 
     (a_1-a_2)^2(a_2-a_3)^2(a_3-a_1)^2, 
$$
respectively, with mutually distinct $ a_1,a_2 $, and $ a_3 $.
The file size of a defining equation of this interior curve
is about 200Kb as a text file. See also \cite{fuji-cmft}.

Thus, it is not so easy to obtain the defining equation of
the interior curve by calculating the envelope of $ B $
of degree greater than 4.

\bigskip

Next, we consider the geometrical properties of 
Blaschke products outside the unit disk.
Let $ B $ be a canonical Blaschke product of degree $ d $.
For $ \lambda\in\partial\mathbb{D} $, let $ L_{\lambda} $ 
be the set of $ d $ lines tangent to 
$ \partial\mathbb{D} $ at the $ d $
preimages of $ \lambda\in\partial\mathbb{D} $ by $ B $.
Here, we denote by $ E_B $
the trace of the intersection points of 
each two elements in $ L_{\lambda} $ as $ \lambda $ ranges over
the unit circle.
We call the trace $ E_B $ the {\it exterior curve associated with $ B $}.

In \cite{fuji-circum}, we obtained the following.

\begin{theorem}[\cite{fuji-circum}]\label{thm:algd}
  Let $ B $ be a canonical Blaschke product of degree $ d $.
  Then, the exterior curve $ E_B $ is an algebraic curve of degree 
  at most $ d-1 $.
\end{theorem}

The proof of Theorem \ref{thm:algd} is already described 
in \cite{fuji-circum}, but we will give an outline proof 
in section \ref{sec:2} in order to provide the defining equation of $ E_B $.

The following result comes to mind when we pay attention to the
degree of $ E_B $.
However, we remark that the degree of the exterior curve may degenerate
to less than $ d-1 $ (see Remark \ref{rem:degree}).

\begin{theorem}[Siebeck \cite{siebeck}]\label{thm:marden2}
  The zeros of the function
  $ 
      F(z)=\sum_{j=1}^{d}\frac{m_j}{z-z_j}\ (m_j\in\mathbb{R^*}), 
  $
  are the foci of the curve of class $ d-1 $ which touches each line-segment
  $ z_jz_k $ in a point dividing the line segment in the ratio $ m_j:m_k $.
\end{theorem}

The main aim of this paper is to explore the relation between the 
geometrical properties of the interior curve and the exterior curve.
As the main theorem in this paper, 
we will show that the following result in section \ref{sec:3}.

\begin{theorem}\label{thm:dual}
  Let $ B $ be a canonical Blaschke product of degree $ d $,
  and $ E^*_B $ the dual curve of the homogenized exterior curve $ E_B $.
  Then, the interior curve is given by
  $$ I_B:\ u^*_B(-z)=0, $$
  where $ u^*_B(z)=0 $ is a defining equation of 
  the affine part of $ E^*_B $.
\end{theorem}

Moreover, as an application of this theorem, 
we construct examples of  Blaschke products
having two ellipses as the interior curve in section \ref{sec:exp}.

\section{Interior and exterior curves}\label{sec:2}

Although, the proof of Theorem \ref{thm:algd} is already described 
in \cite{fuji-circum},
in order to confirm the method of construction of the defining equation,
we provide an outline proof here.

\begin{proof}[Proof of Theorem \textup{\ref{thm:algd}}]
  Let 
  $  \displaystyle 
     B(z)=z\prod_{k=1}^{d-1}\frac{z-a_k}{1-\overline{a_k} z} \
    (a_k\in \mathbb{D}) $
  and written as follows 
  $$
    B(z)=
    \dfrac{z^d-\sigma_1z^{d-1}+\sigma_2z^{d-2}+\cdots+(-1)^{d-1}\sigma_{d-1}z}
     {1-\overline{\sigma_1}z+\cdots+(-1)^{d-1}\overline{\sigma_{d-1}}z^{d-1}},
  $$ 
  where $ \sigma_k $ are the elementary symmetric polynomials on 
  variables $ a_1,\cdots, a_{d-1} $ of degree $ k $  $(k=1,\cdots,d-1) $.
  Let $ \sigma_0=1 $ and $ \sigma_d=0 $.
  Eliminating $ \lambda $ from $ B(z_1)=B(z_2)=\lambda $, we have
  \begin{align*} 
     & \Big(z_1^d-\sigma_1z_1^{d-1}+\cdots+(-1)^{d-1}\sigma_{d-1}z_1\Big)
       \Big(1-\overline{\sigma_1}z_2+\cdots
            +(-1)^{d-1}\overline{\sigma_{d-1}}z_2^{d-1}\Big)\\ 
     & \qquad 
       -\Big(z_2^d-\sigma_1z_2^{d-1}+\cdots+(-1)^{d-1}\sigma_{d-1}z_2\Big)
       \Big(1+\overline{\sigma_1}z_1+\cdots
        +(-1)^{d-1}\overline{\sigma_{d-1}}z_1^{d-1}\Big)\\ 
     & = \sum_{j=1}^d\sum_{k=1}^d 
        (-1)^{j+k}\overline{\sigma_{d-j}}\sigma_{d-k}
               (z_1^kz_2^{d-j}-z_1^{d-j}z_2^k) \\
     & =\sum_{N=1}^d\sum_{K=0}^{N-1} 
            (-1)^{d-N+K}(\sigma_{d-N}\overline{\sigma_{K}}
               -\overline{\sigma_{N}}\sigma_{d-K})
                  (z_1z_2)^K(z_1^{N-K}-z_2^{N-K}) \\ 
     &= (z_1-z_2)\sum_{N=1}^d\sum_{K=0}^{N-1}
            (-1)^{d-N+K}(\sigma_{d-N}\overline{\sigma_{K}}
               -\overline{\sigma_{N}}\sigma_{d-K})
         (z_1z_2)^K\\ 
     &\quad
        \times\Big((z_1+z_2)^{N-K-1}-\gamma_1z_1z_2(z_1+z_2)^{N-K-3}+\cdots
                    +\gamma_M(z_1z_2)^M(z_1+z_2)^R\Big)=0,
  \end{align*}
  where $ R $ is the remainder after dividing $ N-K-1 $ by $ 2 $,
  $$
       M=\frac{N-K-1-R}{2} , \qquad  \gamma_1=N-K-2, 
  $$ 
  and $ \gamma_M  $ is a non-zero coefficient.
  The intersection point $ z $ of two lines $ l_1 $ and $ l_2 $
  satisfies
  \begin{equation}\label{eq:interd}
          z_1z_2=\dfrac{z}{\zb } \quad \mbox{and} \quad
          z_1+z_2=\dfrac{2}{\zb }, 
  \end{equation}
  since each $ l_k\ (k=1,2) $ is a line tangent to the unit circle 
  at a point $ z_k $.
  Note that the intersection point is the point at infinity if and only if
  $ z_1+z_2=0 $.
  Hence, we have
  \begin{align} \notag
      & \sum_{N=1}^d\sum_{K=0}^{N-1}
           (-1)^{d-N+K}\big(\sigma_{d-N}\overline{\sigma_{K}}
               -\overline{\sigma_{N}}\sigma_{d-K}\big)
          z^K\zb^{d-N} \\ \label{eq:algd}
      & \qquad \times \Big( 2^{N-K-1}-2^{N-K-3}\gamma_1z\zb 
            +\cdots + 2^R\gamma_Mz^M\zb^M\Big)=0.
  \end{align}
  This equality gives a defining equation of $ E_B $
  with degree at most $ d-1 $. 
\end{proof}

When the degree is low, we can describe the exterior curve concretely, 
as follows.

\begin{corollary}[\cite{fuji-circum}]  
  Let $ B $ be a canonical Blaschke product of degree $ d $ with
  zeros $ 0 $, $a_1,\cdots $, $a_{d-1} \in\mathbb{D} $.
  \begin{itemize}
     \item For a canonical Blaschke product of degree $ 2 $ with zeros
           $ 0 $ and $ a_1 (\neq0) $, 
           the exterior curve is the line $ \aone z+a_1\zb-2=0 $.
     \item For $ d=3 $, the exterior curve is either an ellipse, 
           a circle, a parabola, or a hyperbola. 
          \begin{equation}\label{env3}
             \aone \atwo z^2
             -(|a_1a_2|^2-|a_1+a_2|^2+1)z\zb +a_1a_2\zb^2
            -2(\aone +\atwo )z-2(a_1+a_2)\zb +4=0.
          \end{equation}
     \item For  $ d=4 $, the defining equation of the exterior curve
           is written as
         \begin{align} \notag
           & \sbar3z^3+(\s1\sbar2-\s2\sbar3-\sbar1)z^2\zb
             -(\s1-\s2\sbar1+\s3\sbar2)z\zb^2+\s3\zb^3 \\ \label{eq:env4gene}
           & -2\sbar2z^2-(2\s1\sbar1-2\s3\sbar3-4)z\zb
             -2\s2\zb^2+4\sbar1z+4\s1\zb-8=0,
        \end{align}%
         where $ \sigma_{k} $ are the elementary symmetric polynomials on
         three variables $ a_1,a_2,a_3 $ of degree $ k \ (k=1,2,3) $, i.e.
         $$ 
             \s1=a_1+a_2+a_3,\ \s2=a_1a_2+a_1a_3+a_2a_3 \quad \mbox{and}\quad
             \s3=a_1a_2a_3. 
         $$
   \end{itemize}
\end{corollary}

Even if we use symbolic computation systems, it is hard to calculate 
the defining equation of the interior curve for $ d=5 $.
However, we can obtain the exterior curve as follows.

\begin{corollary}
  For a canonical Blaschke product of degree 5, 
  the defining equation of the exterior curve 
  \eqref{eq:algd} is written as
\begin{align} \notag
%deg4
  & \sbar4z^4+(\s1\sbar3-\sbar2-\s2\sbar4)z^3\zb
    -(\s1\sbar1-\s2\sbar2+\s3\sbar3-\s4\sbar4-1)z^2\zb^2 \\ \notag
  & +(\s3\sbar1-\s4\sbar2-\s2)z\zb^3+\s4\zb^4
    -2\sbar3z^3+2(2\sbar1-\s1\sbar2+\s3\sbar4)z^2\zb\\ \notag
  & -2(\s2\sbar1-2\s1-\s4\sbar3)z\zb^2
    -2\s3\zb^3
    +4\sbar2z^2+4(\s1\sbar1-\s4\sbar4-3)z\zb \\
 & +4\s2\zb^2 -8\sbar1z-8\s1\zb+16=0,
\end{align}
  where $ \sigma_{k} $ are the elementary symmetric polynomials on
  four variables $ a_1,\cdots,a_4 $ of degree $ k \ (k=1,\cdots,4) $.
\end{corollary}

\begin{remark}\label{rem:degree}
   For $ d=4 $, the degree of the exterior curve \eqref{eq:env4gene}
   is not greater than $ 2 $ if and only if the Blaschke product has a
   double zero point at the origin and the sum of the other zero 
   points equals to 0.
   While, the degree of the defining equation of 
   the exterior curve $ E_B $ is always $ 4 $ for every $ B $
   of degree 5.

   Even though we can obtain the defining equation of the exterior 
   curve concretely for $ d\geq 6$, 
   We abandon to describe it. 
   Because the size of the equation is relatively large.

\end{remark}

\begin{figure}[H]
  \centerline{
    \fbox{\includegraphics[width=0.4\linewidth]{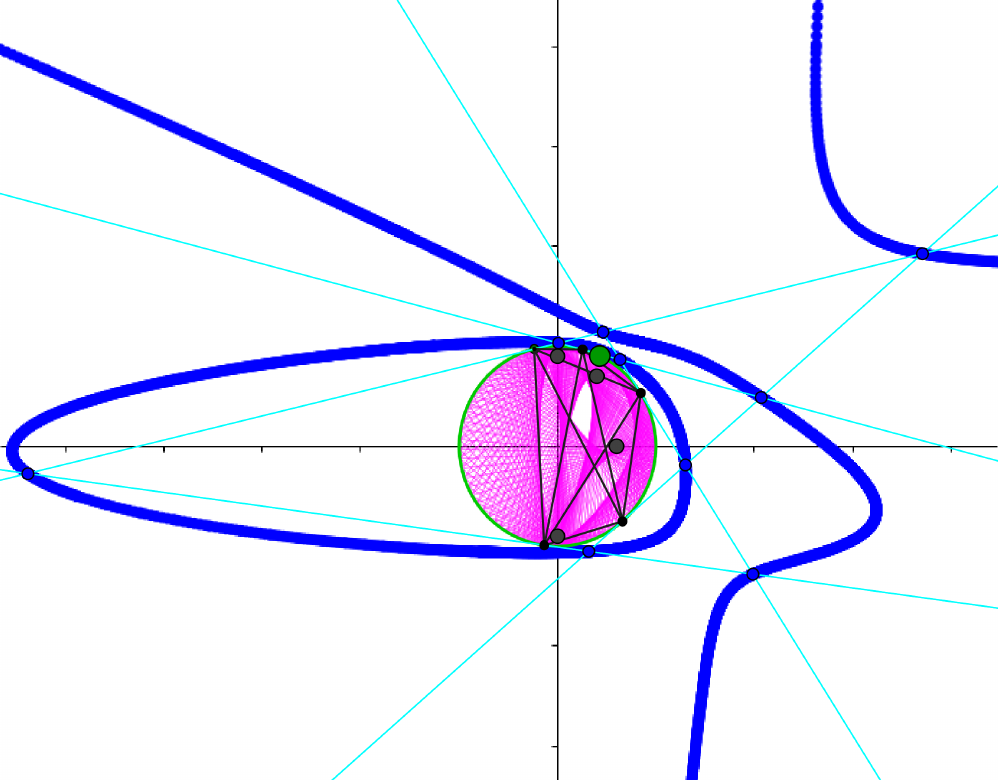}}\qquad
    \fbox{\includegraphics[width=0.4\linewidth]{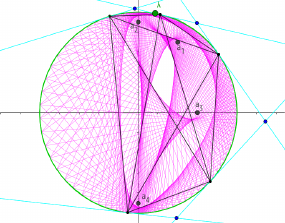}}
   }

  \centerline{
    \fbox{\includegraphics[width=0.4\linewidth]{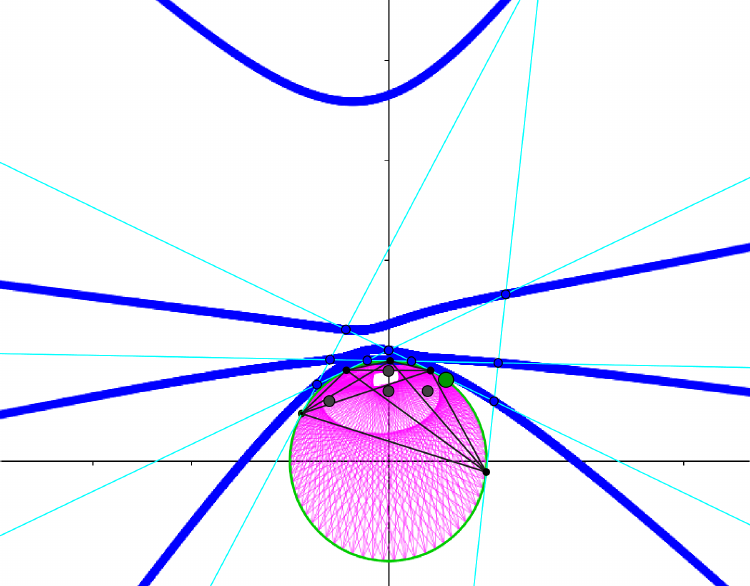}}\qquad
    \fbox{\includegraphics[width=0.4\linewidth]{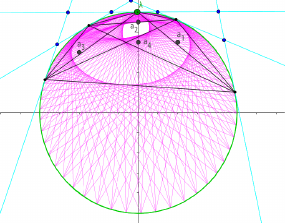}}
  }
  \caption{The thick curves indicate the exterior curves 
           for canonical Blaschke products with non-zero zero points 
           $ (a_1,a_2,a_3,a_4)= (0.4+0.7i,0.9i,0.6,-0.9i) $ (upper) and 
           $ (0.4+0.7i,0.9i,-0.6+0.6i,0.7i)$ (lower).
           The envelope in each right figure is 
           the interior curve corresponding to the left figure.}
  \label{pic:deg5-gene}
\end{figure}

\section{Proof of Theorem \ref{thm:dual}}\label{sec:3}

The affine part of the projective space
$ \mathbb{P}_2(\mathbb{R}) $ can be identified
with the complex plane $ \mathbb{C} $.

Recall that the dual curve $ C^* $ of $ C\subset \mathbb{P}_2(\mathbb{R}) $
is defined by
$$ C^*=\{L\in\mathbb{P}_2^*(\mathbb{R})\,;\,  L \mbox{ is a line tangent to } C
                       \mbox{ at some } p\in C\}. $$

\begin{proof}[Proof of Theorem \ref{thm:dual}]
  Let $ z' $ and $ z'' $ are two preimages for some 
  $ \lambda\in\partial\mathbb{D} $,
  and $ \ell $ is the line joining $ z' $ and $ z''$.
  Let $ \zeta $ be the intersection point of two lines tangent to
  the unit circle at the points $ z' $ and $ z'' $ (cf. Figure \ref{pic:pp}).
  Therefore the point $ \zeta $ is the pole and the line $ \ell $ 
  is its polar with respect to the unit circle.

\begin{figure}[htbp]
    \centerline{\includegraphics[width=0.3\linewidth]{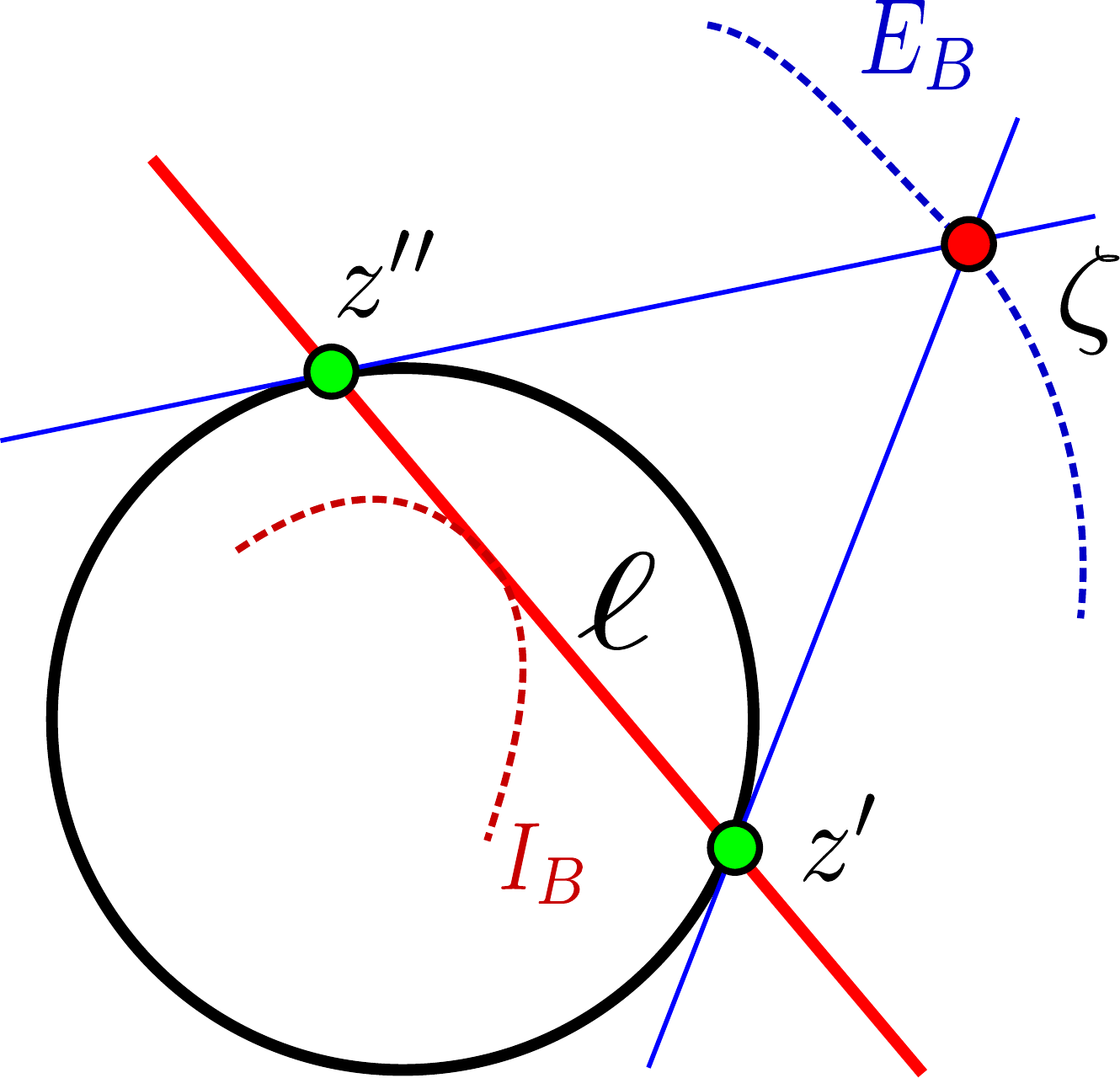}}
    \caption{The point $ \zeta $ is the pole and the line $ \ell $ 
             is its polar with respect to the unit circle.}
    \label{pic:pp}
\end{figure}

  Then, the equation of $ \ell $ is written as
  \begin{equation}\label{eq:ell}
          z+z'z'' \zb=z'+z''.
  \end{equation}
  The intersection point $ \zeta $ satisfies
  \begin{equation}\label{eq:pole}
      z'+z''=\frac2{\overline{\zeta}} \mbox{\qquad and\qquad }
      z'z''=\frac{\zeta}{\overline{\zeta}}.
  \end{equation}
  Substituting \eqref{eq:pole} into \eqref{eq:ell},
  the line $ \ell $ is written by the data of $ \zeta $ as follows,
  $$ 
      \overline{\zeta}z+\zeta\overline{z}=2. 
  $$
  Substituting $ \zeta=\alpha+\beta i $ and $ z=x+yi $ 
  into the above equality again, 
  the line $ \ell $ is expressed as the line on the real $ xy $-plane,
  $ \alpha x+\beta y-1=0. $
  Therefore the line 
  $ \ell\subset\mathbb{C}\subset \mathbb{P}_2(\mathbb{R}) $
  corresponds to the point 
  $ (-\alpha:-\beta:1) \in\mathbb{P}_2^*(\mathbb{R}) $, 
  and this point corresponds to the point $ -\zeta\in\mathbb{C} $.

  \smallskip

  Hence, the assertion is obtained from the fact that
  the family of all tangent lines of the interior curve $ I_B $ 
  coincides with the family of lines 
  $ \mathcal{L}_B=\{ \ell_{\lambda}\}_{\lambda} $.
\end{proof}

Equivalently, the converse also holds.

\begin{corollary}
  Let $ B $ be a canonical Blaschke product of degree $ d $,
  and $ I^*_B $ be the dual curve of the homogenized interior curve $ I_B $.
  Then, the exterior curve is given by
  $$ E_B:\ v^*_B(-z)=0, $$
  where $ v^*_B(z)=0 $ is a defining equation of 
  the affine part of $ I^*_B $.
\end{corollary}

\begin{remark}
  As we mentioned in section $ 1 $, for $ d=3 $, the ellipse \eqref{eq:DGM}
  corresponds to the inner ellipse of Poncelet's theorem.
  For $ d\geq 4 $, Theorems {\rm \ref{thm:algd}} and {\rm \ref{thm:dual}}
  provides the defining equation of the envelope of the family of lines 
  $ \{\ell_{\lambda}\}_{\lambda\in\mathbb{D}} $,
  where $ \ell_\lambda $ is the set of all segments joining each distinct two
  preimages in $ B^{-1}(\lambda) $.
  Here, we remark that $ \ell_\lambda $ includes diagonals of the 
  $ d $-sided polygon with vertices at $ B^{-1}(\lambda) $.
  In general, the defining equation of this envelope is not always reducible, 
  but the ``outermost part'' of the curve gives the so-called Poncelet curve
  associated with the Blaschke product.
  For instance, see {\rm \cite{mirman}} 
  and {\rm \cite[Definition 5.1 and Theorem 5.2]{daepp2015}} 
  for details about definitions and related topics of Poncelet curve.
\end{remark}

\section{Examples}\label{sec:exp}

For a Blaschke product $ B $ of degree $ 4 $, 
the interior curve $ I_B $ is an ellipse if and only if 
$ B $ is a composition of two Blaschke products of degree $ 2 $.
See \cite{fuji-circum}, 
and also see \cite{gorkin2} for the relationship between this result
and the numerical range of shift perators.

Here, as an appliation of Theorem \ref{thm:dual},
we construct a Blaschke product of degree 5
whose interior curve is a union of two ellipses.

\begin{example}
  Let
  $$ 
      B_{a,b}(z)=z\frac{z^2-a}{1-a z^2}\frac{z^2-b}{1-b z^2}
      \qquad (0<a,b<1), 
  $$
  where $ a,\,b $ satisfy the equality 
  $ a^3b^3-2a^2b^2-(b^2+a^2)+3ab=0 $.

  In this case, $ E_{B_{a,b}} $ is given as follows,
  \begin{align} \notag
    E_{B_{a,b}} : & \Big(a(b+1)^2x^2+a(b-1)^2y^2-4b\Big) \\ \label{eq:ex1}
                &  \Big((a^2b^3-ab^2+2b^2+3b-a)x^2+(a^2b^3-ab^2-2b^2+3b-a)y^2
                   -4b\Big)=0,
  \end{align}
  where $ z=x+iy $.

  Therefore, the exterior curve is a union of two ellipses for every
  $ B_{a,b} $.
  Then, the interior curve is also a union of two ellipses
  because the dual curve of an irreducible conic is also 
  an irreducible conic
  and the interior curve is a compact curve in $ \mathbb{D} $. 
  See Figure {\rm \ref{pic:elip-elip1}}.

\begin{figure}[htbp]
  \centerline{
    \fbox{\includegraphics[width=0.45\linewidth]{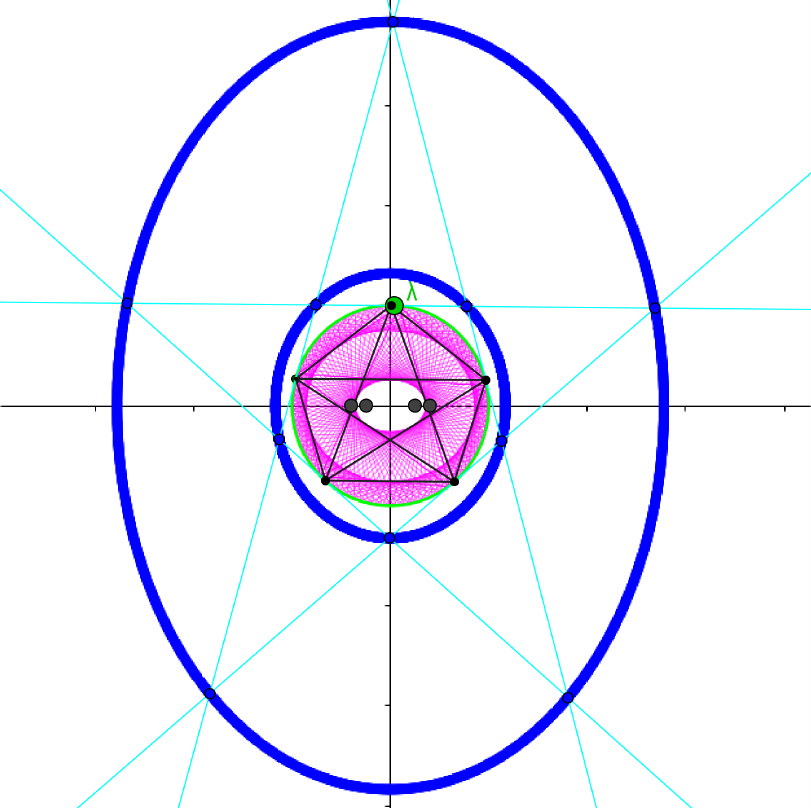}}\quad
    \fbox{\includegraphics[width=0.45\linewidth]{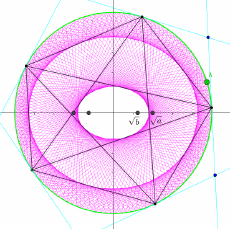}}
   }
  \caption{The thick curve indicates the exterior curve $ E_{B_{a,b}} $,
           where $ (a,b)=(0.16,0.0616)$. The envelope in 
           the right figure is the interior curves correspond 
           to the left figure.}
  \label{pic:elip-elip1}
\end{figure}

  In fact, $ I_{B_{a,b}} $ is given by, 
  $$
      I_{B_{a,b}} :  
      \Big(\frac{4b}{a(b+1)^2}x^2+\frac{4b}{a(b-1)^2}y^2-1\Big)
      \Big(\frac{4a}{b(a+1)^2}x^2+\frac{4a}{b(a-1)^2}y^2-1\Big)=0.
  $$
  The two foci are $ \pm\sqrt{a} $ (the first factor) 
  and $ \pm\sqrt{b} $ (the second factor).
\end{example}

\begin{example}
  Let
  $$ 
      B_c(z)=z\Big(\frac{z-\frac14}{1-\frac14 z}\Big)^2
           \Big(\frac{z-c}{1-c z}\Big)^2 \qquad
          (0<c<1), 
  $$
  where $ c $ is a solution of $ c^3-72c^2+48c-4=0 $.
  There are two possibilities of $ c $,
  $$ c\approx 0.0976036,\quad \mbox{or} \quad   c\approx 0.5745591. $$

  In this case, $ E_{B_c} $ is given as follows,
  \begin{align} \notag
     E_{B_c} : & \Big(4z^2+(-225\zb c+8\zb-64)z+4\zb^2
              -64\zb+256\Big) \\ \label{eq:elipelip}
           & \Big(16c^2z^2+(-257\zb c^2+(272\zb -64)c-64\zb )z
                 +16\zb^2c^2-64\zb c+64\Big)=0.
  \end{align}
  In the case of $ c\approx 0.0976036 $,
  \eqref{eq:elipelip} is a union of two ellipses.
  See {\rm Figure \ref{pic:elip-elip}}.

\begin{figure}[htbp]
  \centerline{
    \fbox{\includegraphics[width=0.45\linewidth]{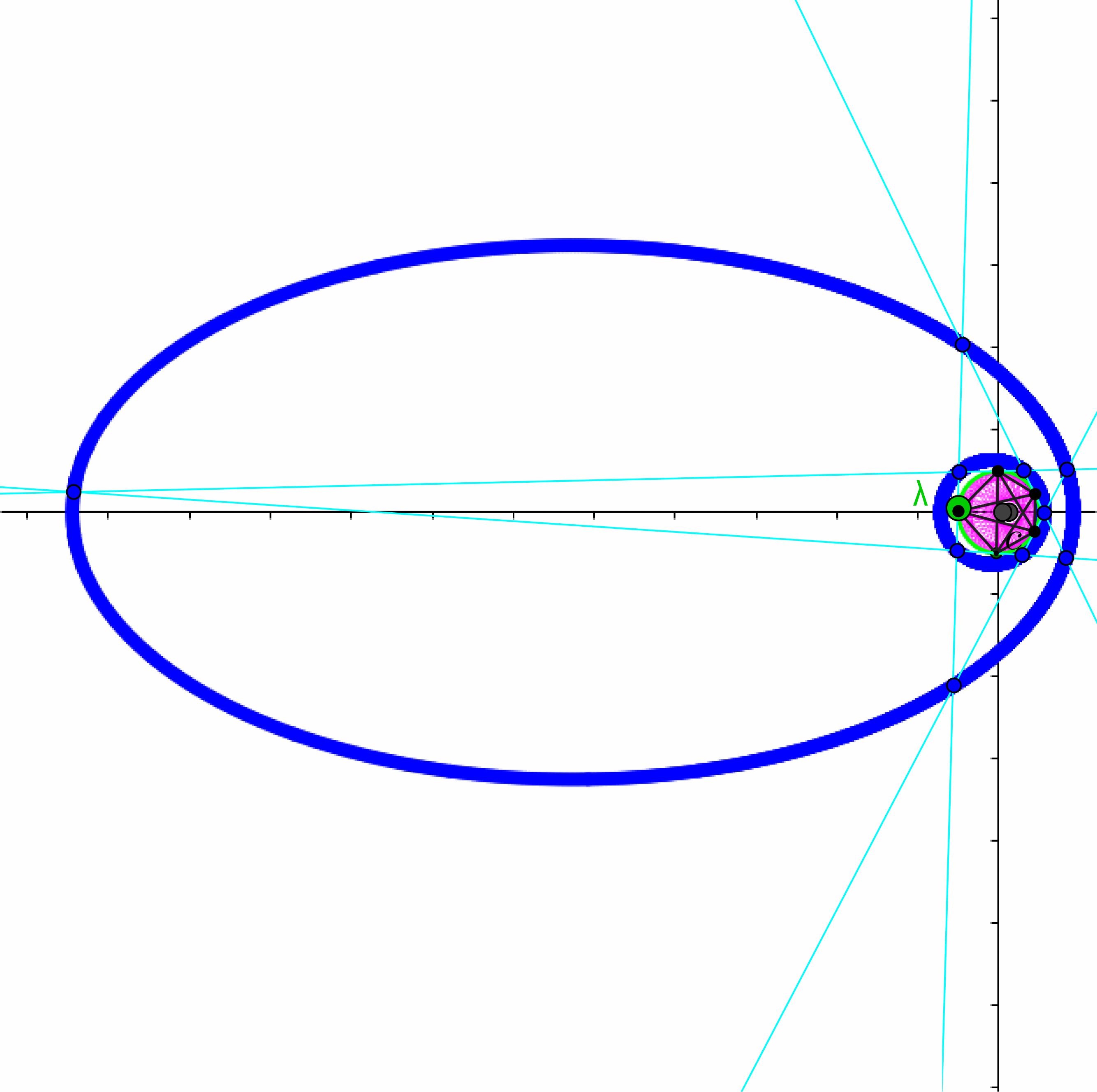}}\quad
    \fbox{\includegraphics[width=0.45\linewidth]{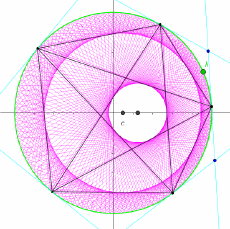}}
   }
  \caption{The thick curve indicates the exterior curve 
           for $ B_c $ with $ c\approx 0.0976036 $.
           The envelope  in the right figure is the interior curve.}
  \label{pic:elip-elip}
\end{figure}

\begin{figure}[htbp]
  \centerline{
    \fbox{\includegraphics[width=0.45\linewidth]{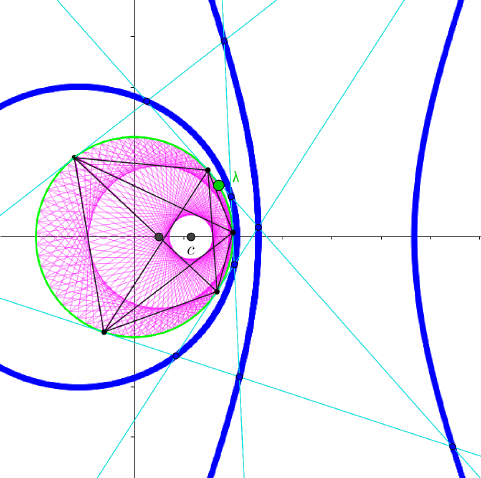}}\quad
    \fbox{\includegraphics[width=0.45\linewidth]{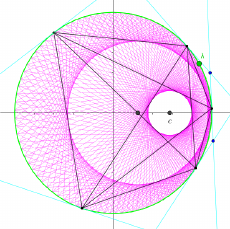}}
   }
    \caption{The thick curve indicates the exterior curve
             for $ B_c $ with $ c\approx 0.5745591 $.
             The envelope  in the right figure is the interior curve.}
  \label{pic:elip-hyp}
\end{figure}

  The other case, \eqref{eq:elipelip} is a union of an ellipse and
  a hyperbola. See {\rm Figure \ref{pic:elip-hyp}}.
  In any case, the interior curve should be a union of two ellipses.

  In fact, the interior curve is the union of the following two circles,
  $$ 
      \Big|z-\frac14\Big|=\frac{15}{16}\sqrt{c}, \quad \mbox{and}\quad
      \big|z-c\big|=\frac18(17c-8).
  $$
\end{example}

\bibliographystyle{alpha}
\bibliography{fuji-ref}

\begin{thebibliography}{DGSS15}

\bibitem[DGM02]{daepp}
U.~Daepp, P.~Gorkin, and R.~Mortini.
\newblock {Ellipses and finite Blaschke products}.
\newblock {\em {Amer. Math. Monthly}}, 109:785--794, 2002.

\bibitem[DGSS15]{daepp2015}
U.~Daepp, P.~Gorkin, A.~Shaffer, and B.~Sokolowsky.
\newblock {Decomposing finite Blaschke products}.
\newblock {\em {J. Math. Anal. Appl.}}, 426:1201--1216, 2015.

\bibitem[Fla08]{flatto}
L.~Flatto.
\newblock {\em {Poncelet's Theorem}}.
\newblock {Amer. Math. Soc.}, {Providence}, 2008.

\bibitem[Fuj13]{fuji-cmft}
M.~Fujimura.
\newblock {Inscribed ellipses and Blaschke products}.
\newblock {\em Comput. Methods Funct. Theory}, 13:557--573, 2013.

\bibitem[Fuj17]{fuji-circum}
M.~Fujimura.
\newblock {Blaschke products and circumscribed conics}.
\newblock {\em {Comput. Methods Funct. Theory}}, 17:635--652, 2017.

\bibitem[GS11]{gorkin}
P.~Gorkin and E.~Skubak.
\newblock {Polynomials, ellipses, and matrices: two questions, one answer}.
\newblock {\em {Amer. Math. Monthly}}, 118:522--533, 2011.

\bibitem[GW17]{gorkin2}
P.~Gorkin and N.~Wagner.
\newblock {Ellipses and compositions of finite Blaschke products}.
\newblock {\em J. Math. Anal. Appl.}, 445:1354--1366, 2017.

\bibitem[Mar66]{marden}
M.~Marden.
\newblock {\em {Geometry of Polynomials, 2nd edition}}.
\newblock {Amer. Math. Soc.}, 1966.

\bibitem[Mas13]{Inner}
J.~Mashreghi.
\newblock {\em {Derivatives of Inner Functions}}.
\newblock Springer-Verlag, New York, 2013.

\bibitem[Mir03]{mirman}
B.~Mirman.
\newblock {UB-matrices and conditions for Poncelet polygon to be closed}.
\newblock {\em {Linear Algebra Appl.}}, 360:123--150, 2003.

\bibitem[Sie65]{siebeck}
P.~Siebeck.
\newblock {Ueber eine neue analytische Behandlungsweise der Brennpunkte}.
\newblock {\em J. Reine Angew. Math.}, 64:175--182, 1865.

\end{thebibliography}

\end{document}